\documentclass{amsart}
\usepackage{amssymb}
\newcommand{\PP}{\mathcal B}
\newcommand{\covJ}{\mathrm{cov}{}\kern-2pt{}_{\mathcal J}}

\title{Densities, submeasures and partitions of groups}
\author{Taras Banakh, Igor Protasov, Sergiy Slobodianiuk}
\address{T.Banakh: Ivan Franko National University of Lviv, Ukraine and Jan Kochanowski University in Kielce, Poland.}
\email{t.o.banakh@gmail.com}
\address{I.Protasov and S.Slobodianiuk: Taras Shevchenko National University, Kyiv, Ukraine}
\email{i.v.protasov@gmail.com, slobodianiuk@yandex.ru}

\newcommand{\C}{\mathcal C}
\newcommand{\A}{\mathcal A}
\newcommand{\I}{\mathcal I}
\newcommand{\IN}{\mathbb N}

\newcommand{\cov}{\mathrm{cov}}
\newcommand{\w}{\omega}
\newcommand{\is}{\mathsf{is}}
\newcommand{\us}{\mathsf{us}}

\newcommand{\wis}{\widehat{\mathsf{\i s}}}
\newcommand{\iss}{\mathsf{iss}}

\newcommand{\sis}{\mathsf{sis}}

\newcommand{\wus}{\widehat{\mathsf{us}}}
\newcommand{\e}{\varepsilon}

\newcommand{\conj[1]}{\wr{#1}}
\newcommand{\Ipack}{\mathrm{pack}_\I}
\newcommand{\pack}{\mathrm{pack}}
\newcommand{\IP}{\mathrm{IP}}

\newtheorem{theorem}{Theorem}[section]
\newtheorem{problem}[theorem]{Problem}
\newtheorem{question}[theorem]{Problem}
\newtheorem{proposition}[theorem]{Proposition}
\newtheorem{corollary}[theorem]{Corollary}

\theoremstyle{definition}
\newtheorem{remark}[theorem]{Remark}

\newtheorem{example}[theorem]{Example}

\begin{document}
\begin{abstract}
In 1995 in Kourovka notebook the second author asked the following problem: {\em is it true that for each partition $G=A_1\cup\dots\cup A_n$ of a group $G$ there is a cell $A_i$ of the partition such that $G=FA_iA_i^{-1}$ for some set $F\subset G$ of cardinality $|F|\le n$?} In this paper we survey several partial solutions of this problem, in particular those involving certain canonical invariant densities  and submeasures on groups. In particular, we show that for any partition $G=A_1\cup\dots\cup A_n$  of a group $G$ there are cells $A_i$, $A_j$ of the partition such that
\begin{itemize}
\item $G=FA_jA_j^{-1}$ for some finite set $F\subset G$ of cardinality $|F|\le \max_{0<k\le n}\sum_{p=0}^{n-k}k^p\le n!$;
\item $G=F\cdot\bigcup_{x\in E}xA_iA_i^{-1}x^{-1}$ for some finite sets $F,E\subset G$ with $|F|\le n$;
\item $G=FA_iA_i^{-1}A_i$ for some finite set $F\subset G$ of cardinality $|F|\le n$;
\item the set $(A_iA_i^{-1})^{4^{n-1}}$ is a subgroup of index $\le n$ in $G$.
\end{itemize}
The last three statements are derived from the corresponding density results.
\end{abstract}
\maketitle

\section{Introduction}

In this paper we survey partial solutions to the following open problem posed by I.V.Protasov in 1995 in the Kourovka notebook \cite[Problem 13.44]{Kourovka}.

\begin{problem}\label{prob1} Is it true that for any finite partition $G=A_1\cup\dots\cup A_n$ of a group $G$ there is a cell $A_i$ of the partition and a subset $F\subset G$ of cardinality $|F|\le n$ such that $G=FA_iA_i^{-1}$?
\end{problem}

In \cite{Kourovka} it was observed that this problem has simple affirmative solution for amenable groups (see Theorem~\ref{t4.3} below).

Problem~\ref{prob1} is a partial case of its ``idealized'' $G$-space version.
Let us recall that a $G$-space is a set $X$ endowed with a left action $\alpha:G\times X\to X$, $\alpha:(g,x)\mapsto gx$, of a group $G$. Each group $G$ will be considered as a $G$-space endowed with the left action $\alpha:G\times G\to G$, $\alpha:(g,x)\mapsto gx$.

A non-empty family $\I$ of subsets of a set $X$ is called a {\em Boolean ideal} if for any $A,B\in\I$ and $C\subset X$ we get $A\cup B\in\I$ and $A\cap C\in\I$. A Boolean ideal $\I$ on a set $X$ will be called {\em trivial\/} if it coincides with the Boolean ideal $\PP(X)$ of all subsets of $X$. By $[X]^{<\w}$ we shall denote the Boolean ideal consisting of all finite subsets of $X$. A Boolean ideal $\I$ on a $G$-space $X$ is called {\em $G$-invariant} if for any $A\in\I$ and $g\in G$ the shift $gA$ of $A$ belongs to the ideal $\I$. By an {\em ideal $G$-space} we shall understand a pair $(X,\I)$ consisting of a $G$-space $X$ and a non-trivial $G$-invariant Boolean ideal $\I\subset\PP(X)$.

For an ideal $G$-space $(X,\I)$ and a subset $A\subset X$ the set
$$\Delta_\I(A)=\{x\in G:A\cap xA\notin\I\}\subset G$$
will be called the {\em $\I$-difference set} of $A$. It is not empty if and only if $A\notin\I$.

For a non-empty subset $A\subset G$ of a group $G$ its {\em covering number} is defined as
$$\cov(A)=\min\{|F|:F\subset G,\;G=FA\}.$$
More generally, for a Boolean ideal $\mathcal J\subset\PP(G)$ on a group $G$ and a non-empty subset $A\subset G$ let
$$\covJ(A)=\min\{|F|:F\subset G,\;\;G\setminus FA\in\mathcal J\}$$be the {\em $\mathcal J$-covering number} of $A$.

Observe that for the smallest Boolean ideal $\I=\{\emptyset\}$ on a group $G$ and a subset $A\subset G$ the $\I$-difference set $\Delta_\I(A)$ is equal to $AA^{-1}$.
That is why Problem~\ref{prob1} is a partial case of the following more general

\begin{problem}\label{prob2} Is it true that for any finite partition $X=A_1\cup\dots\cup A_n$ of an  ideal $G$-space $(X,\I)$ some cell $A_i$ of the partition has
\begin{itemize}
\item $\cov(\Delta_\I(A_i))\le n$?
\item $\covJ(\Delta_\I(A_i))\le n$ for some non-trivial $G$-invariant Boolean ideal $\mathcal J$ on the acting group $G$?
\end{itemize}
\end{problem}

Problems~\ref{prob1} and \ref{prob2} can be reformulated in terms of the functions $\Phi_{G}(n)$, $\Phi_{(X,\I)}(n)$ defined as follows. For an ideal $G$-space $\mathbf X=(X,\I)$ and a (real) number $n\ge 1$ denote by $$X/n=\{\C\subset\PP(X):|\C|\le n,\;\cup\C=X\}$$ the family of all at most $n$-element covers of $X$ and put $\Phi_{\mathbf X}(n)=\sup_{\C\in X/n}\min_{C\in\C}\cov(\Delta_\I(C))$.
For each $G$-space $X$ we shall write $\Phi_{X}(n)$ instead of $\Phi_{(X,\{\emptyset\})}(n)$ (in this case we identify $X$ with the ideal $G$-space $(X,\{\emptyset\})$. In particular, for each group $G$ we put $\Phi_G(n)=\sup_{\A\in G/n}\min_{A\in\A}\cov(AA^{-1}).$

For every ideal $G$-space $\mathbf X=(X,\I)$ the definition of the number $\Phi_{\mathbf X}(n)$ implies that for any partition $X=A_1\cup \dots\cup A_n$ of $X$ there is a cell $A_i$ of the partition with  $\cov(\Delta_\I(A_i))\le \Phi_{\mathbf X}(n)$. This fact allows us to reformulate and extend  Problem~\ref{prob2} as follows.

\begin{problem}\label{prob3} Study the growth of the function $\Phi_{\mathbf X}(n)$ for a given ideal $G$-space $\mathbf X=(X,\I)$. Detect ideal $G$-spaces $\mathbf X$ with $\Phi_{\mathbf X}(n)\le n$ for all $n\in\IN$.
\end{problem}

Problems~\ref{prob1}--\ref{prob3} have many partial solutions, which can be divided into three categories corresponding to methods used in these solutions.

The first category contains results giving upper bounds on the function $\Phi_{\mathbf X}(n)$ proved by a combinatorial approach first exploited by Protasov and Banakh in \cite[\S12]{PB} and then refined by Erde \cite{Erde}, Slobodianiuk \cite{Slob} and Banakh, Ravsky, Slobodianiuk \cite{BRS}. These results are surveyed in Section~\ref{s2}.
The first non-trivial result proved by this approach was the upper bound $\Phi_{\mathbf X}(n)\le 2^{2^{n-1}-1}$ proved in Theorem 12.7 of \cite{PB} for groups $G$ endowed with the smallest ideal $\I=\{\emptyset\}$ and generalized later by Slobodianiuk (see \cite[4.2]{Prot}) and Erde \cite{Erde} to infinite groups $G$ endowed with the ideal $\I$ of finite subsets of $G$. Later Slobodianiuk \cite{Slob} using a tricky algorithmic approach, improved this upper bound to $\Phi_{\mathbf X}(n)\le n!$ for any ideal $G$-space $\mathbf X$. This algorithmic approach was developed by Banakh, Ravsky and Slobodianiuk \cite{BRS}
who proved the upper bound  $\Phi_{\mathbf X}(n)\le \varphi(n+1):=\max_{1<k<n}\sum_{i=0}^{n-k}k^i\le n!$, which is the best general upper bound on the function $\Phi_{\mathbf X}(n)$ available at the moment. The  function $\varphi(n)$ grows faster than any exponent $a^n$ but slower than the sequence of factorials $n!$. Unfortunately, it grows much faster than the identity function $n$ required in Problem~\ref{prob3}.

The second category of partial solutions of Problems~\ref{prob1}--\ref{prob3} exploits various $G$-invariant submeasures $\mu:\PP(X)\to[0,1]$ on a $G$-space $X$ and is presented in Sections~\ref{s3}--\ref{s9}. Given such a submeasure $\mu$, for any partition $X=A_1\cup\dots\cup A_n$ of a $G$-space $X$, we use the subadditivity of $\mu$ to select a cell $A_i$ with submeasure $\mu(A_i)\ge \frac1n$ and then use some specific properties of the submeasure $\mu$ to derive certain largeness property of the $\I$-difference set $\Delta_\I(A_i)$ for $G$-invariant Boolean ideals $\I\subset\{A\in\PP(X):\mu(A)=0\}$.

In Section~\ref{s4} we use for this purpose $G$-invariant finitely additive measures (which exist only on amenable $G$-spaces), and prove that for each partition $X=A_1\cup\dots\cup A_n$ of $G$-space with amenable acting group $G$, endowed with a non-trivial $G$-invariant Boolean ideal $\I\subset\PP(G)$, some cell $A_i$ of the partition has $\cov(\Delta_\I(A_i))\le n$, which is equivalent to saying that $\Phi_{(G,\I)}(n)\le n$ for all $n\in\IN$.

In Section~\ref{s5} we apply the extremal density $\is_{12}:\PP(X)\to[0,1]$, $\is_{12}:A\mapsto\inf_{\mu\in P_\w(G)}\sup_{\nu\in P_\w(X)}\mu*\nu(A)$, introduced in \cite{Sol} and studied in \cite{Ban}. On any $G$-space $X$ with amenable acting group $G$ the density $\is_{12}$ is subadditive and coincides with the upper Banach density $d^*$, well-known in Combinatorics of Groups (see e.g., \cite{HS}). Using the density $\is_{12}$ we show that each subset $A\subset X$ with positive density $\is_{12}(A)>0$ has $\cov(\Delta_\I(A))\le1/\is_{12}(A)$. In general, the density $\is_{12}$ is not subadditive, which does not allow to apply it directly to partitions of groups. However, its modification $\sis_{123}$ considered in Section~\ref{s9} is subadditive. Using this feature of the submeasure $\sis_{123}$ we prove that for each partition $X=A_1\cup\dots\cup A_n$ there is a cell $A_i$ of the partition and a finite set $E\subset G$ such that for each $G$-invariant Boolean ideal $\I\subset \{A\in\PP(X):\sis_{123}(A)=0\}$ the set
$\Delta_\I(A_i)^{\conj[E]}=\bigcup_{x\in E}x^{-1}\Delta_\I(A_i)x$ has $\cov(\Delta_I(A_i^{\conj[E]}))\le n$. This implies an affirmative answer to Problem~\ref{prob1} for partitions $G=A_1\cup\dots\cup A_n$ of groups into conjugation-invariant sets
$A_i=A_i^{\conj[G]}$.

The third group of partial solutions of Problems~\ref{prob1} and \ref{prob2} is presented in Section~\ref{s8} and exploits the compact right-topological semigroup $P(G)$ of measures on a group $G$. This approach is developed in a recent paper of Banakh and Fr\c aczyk \cite{BF} where they used minimal measures to prove that for each partition $X=A_1\cup\dots\cup A_n$ of an ideal $G$-space $(X,\I)$ either $\cov(\Delta_\I(A_i))\le n$ for all $i$ or else there is a cell $A_i$ of the partition such that $\cov(\Delta_\I(A_i)\cdot\Delta_\I(A_i))<n$ and $\covJ(\Delta_\I(A_i))<n$ for some $G$-invariant Boolean ideal $\mathcal J\not\ni\Delta_\I(A_i)$ on $G$. Using quasi-invariant idempotent measures, Banakh and Fr\c aczyk proved \cite{BF} that for any partition $G=A_1\cup\dots\cup A_n$ of a group $G$ either $\cov(A_iA_i^{-1})\le n$ for all $i$ or else $\cov(A_iA_i^{-1}\kern-1pt A_i)<n$ for some cell $A_i$ of the partition. We use these facts to prove that for each partition $G=A_1\cup \dots\cup A_n$ of a group $G$ for some cell $A_i$ of the partition, the product $(A_iA_i)^{4^{n-1}}$ is a subgroup of index $\le n$ in $G$.

In Section~\ref{s12n} we apply the density $\is_{12}$ to $\IP^*$-sets and show that for each subset $A$ of a group $G$ with positive density $\is_{21}(A)>0$ the set $AA^{-1}$ is an $\IP^*$-set in $G$ and hence belongs to every idempotent of the compact right-topological semigroup $\beta G$.

In the final Section~\ref{s12} we pose some open problems related to Problems~\ref{prob1}--\ref{prob3}.

\section{Some upper bounds on the function $\Phi_{\mathbf X}(n)$}\label{s2}

In this section we survey some results giving upper bounds on the function $\Phi_{\mathbf X}(n)$, which are proved by combinatorial and algorithmic arguments. Unfortunately, the obtained upper bounds are much higher than the upper bound $n$ required in Problem~\ref{prob3}.

We recall that for an ideal $G$-space $\mathbf X=(X,\I)$ the function $\Phi_{\mathbf X}(n)$ is defined by
$$\Phi_{\mathbf X}(n)=\sup_{\C\in X/n}\min_{A\in\C}\cov(\Delta_\I(A))\;\;\;\mbox{ for $n\in\IN$}.
$$
If $\I=\{\emptyset\}$, then we write $\Phi_X(n)$ instead of $\Phi_{(X,\{\emptyset\})}(n)$. In particular,
for a group $G$,
$$\Phi_{G}(n)=\sup_{\C\in G/n}\min_{A\in\C}\cov(AA^{-1})\;\;\;\mbox{ for $n\in\IN$}.
$$

Historically, the first non-trivial upper bound on the function $\Phi_{\mathbf X}(n)$ appeared in Theorem 12.7 of \cite{PB}.

\begin{theorem}[Protasov-Banakh]\label{t2.1} For any partition $G=A_1\cup\dots\cup A_n$ of a group $G$ some cell $A_i$ of the partition has $\cov(A_iA_i^{-1})\le 2^{2^{n-1}-1}$, which implies that $\Phi_G(n)\le 2^{2^{n-1}-1}$ for all $n\in\IN$.
\end{theorem}

In fact, the proof of Theorem~12.7 from \cite{PB} gives a bit better upper bound than $2^{2^{n-1}-1}$, namely:

\begin{theorem}[Protasov-Banakh]\label{t2.2} For any partition $G=A_1\cup\dots\cup A_n$ of a group $G$ some cell $A_i$ of the partition has $\cov(A_iA_i^{-1})\le u(n)$ where $u(1)=1$ and $u(n+1)=u(n)(u(n)+1)$ for $n\in\IN$. Consequently, $\Phi_G(n)\le u(n)$ for all $n\in\IN$.
\end{theorem}

Observe that the sequence $u(n)$ has double exponential growth
$$2^{2^{n-2}}\le u(n)\le 2^{2^{n-1}-1}\mbox{ \ for \ $n\ge 2$}.$$
The method of the proof of Theorem~\ref{t2.2} works also for ideal $G$-spaces (see \cite{Erde}) which allows us to obtain the following generalization of Theorem~\ref{t2.2}.

\begin{theorem}\label{t2.3} For any partition $X=A_1\cup\dots\cup A_n$ of an ideal $G$-space $\mathbf X=(X,\I)$ some cell $A_i$ of the partition has $\cov(\Delta_\I(A_i))\le u(n)\le 2^{2^{n-1}-1}$ where $u(1)=1$ and $u(n+1)=u(n)(u(n)+1)$ for $n\in\IN$.
This implies $\Phi_{\mathbf X}(n)\le u(n)$ for all $n\in\IN$.
 \end{theorem}

In \cite{Slob} S.Slobodianiuk invented a method allowing to replace the upper bound $u(n)$ in Theorem~\ref{t2.3} by a function $\hbar(n)$, which grows slower that $n!$. To define the function $\hbar(n)$ we need to introduce some notation.

Given an natural number $n$ denote by $\w^n$ the set of all functions defined on the set $n=\{0,\dots,n-1\}$ and taking values in the set $\w$ of all finite ordinals. The set $\w^n$ is endowed with a partial order in which $f\le g$ if and only if $f(i)\le g(i)$ for all $i\in n$.
For an index $i\in n$ by $\chi_i:n\to\{0,1\}$ we denote the characteristic function of the singleton $\{i\}$, which means that $\{i\}=\chi_i^{-1}(1)$.

For subsets $A_0,\dots,A_{n-1}\subset\w^n$ let
$$\sum_{i\in n}A_i=\Big\{\sum_{i\in n}a_i:\forall i\in n\;\;a_i\in A_i\Big\}$$
be the pointwise sum of these sets.

Now given any function $h\in\w^n$ we define finite subsets $h^{[m]}(i)\subset \w^n$, $i\in n$, $m\in\w$, by the recursive formula:
\begin{itemize}
\item $h^{[0]}(i)=\{\chi_i\}$;
\item $h^{[m+1]}(i)=h^{[m]}(i)\cup\{x-x(i)\chi_i:x\in\sum_{i\in n}h^{[m]}(i),\;x\le h\}$ for $m\in\w$.
\end{itemize}
A function $h\in\w^n$ is called {\em $0$-generating} if the constant zero function $\mathbf 0:n\to\{0\}$ belongs to the set $h^{[m]}(i)$ for some $i\in n$ and $m\in\w$.

Now put $\hbar(n)$ be the smallest number $c\in\w$ for which the constant function $h:n\to\{c\}$ is 0-generating. The function $\hbar(n)$ coincides with the function $s_{-\infty}(n)$ considered and evaluated in \cite{BRS}. The proof of the following theorem (essentially due to Slobodianiuk) can be found in \cite{BRS}.

\begin{theorem}[Slobodianiuk]\label{t2.4} For any partition $X=A_1\cup\dots\cup A_n$ of an ideal $G$-space $\mathbf X=(X,\I)$ some cell $A_i$ of the partition has $\cov(\Delta_\I(A_i))\le\hbar(n)$. This implies that $\Phi_{\mathbf X}(n)\le \hbar(n)$ for all $n\in\IN$.
\end{theorem}

The growth of the sequence $\hbar(n)$ was evaluated in \cite{BRS} with help of the functions
$$\varphi(n)=\max_{0<k<n}\sum_{i=0}^{n-k-1}k^i=\max_{1<k<n}\frac{k^{n-k}-1}{k-1}\mbox{ \ and \ } \phi(n)=\sup_{1<x<n}\frac{x^{n-x}-1}{x-1}.$$

\begin{theorem}[Banakh-Ravsky-Slobodianiuk]\label{t2.5n} For every $n\ge 2$ we get
$$\varphi(n)\le\phi(n)<\hbar(n)\le \varphi(n+1)\le \phi(n+1).$$
\end{theorem}

The growth of the function $\phi(n)$ was evaluated in \cite{BRS} with help of the Lambert W-function, which is inverse  to the function $y=xe^x$. So, $W(y)e^{W(y)}=y$ for each positive real numbers $y$.
It is known \cite{Lambert} that at infinity the Lambert W-function $W(x)$ has asymptotical growth
$$W(x)=L-l+\frac{l}{L}+\frac{l(-2+l)}{2L^2}+\frac{l(6-9l+2l^2)}{6L^3}+\frac{l(-12+36l-22l^2+3l^3)}{12L^4}+
O\Big[\Big(\frac{l}{L}\Big)^5\Big]$$
where $L=\ln x$ and $l=\ln\ln x$.

The growth of the sequence $\phi(n+1)$ was evaluated in \cite{BRS} as follows:

\begin{theorem}[Banakh]\label{bound:phi} For every $n>50$
$$nW(ne)-2n+\frac{n}{W(ne)}+\frac{W(ne)}{n}<\ln \phi(n+1)<nW(ne)-2n+\frac{n}{W(ne)}+\frac{W(ne)}{n}+\frac{\ln\ln (ne)}{n}$$
and hence $$\ln\phi(n+1)=n\ln n-n-n\big(\ln\ln(n)+O(\tfrac{\ln\ln n}{\ln n})\big).$$
\end{theorem}

It light of Theorem~\ref{bound:phi}, it is interesting to compare the growth of the sequence $\phi(n)$ with the growth of the sequence $n!$ of factorials. Asymptotical bounds on $n!$ proved  in \cite{factorial} yield the following lower and upper bounds on the logarithm $\ln n!$ of $n!$:
$$n\ln n-n+\frac12\ln n+\frac{\ln 2}2+\frac1{12n+1}<\ln n!<n\ln n-n+\frac{\ln n}n+\frac12\ln n+\frac{\ln 2}{2}+\frac1{12n}.$$
Comparing these two formulas, we see that the sequence $\phi(n)$ as well as $\hbar(n)$ grows faster than any exponent $a^n$, $a>1$, but slower than the sequence $n!$ of factorials.

In fact the upper bound $\varphi(n+1)\le n!$ can be easily proved by induction:

\begin{proposition} Each ideal $G$-space $\mathbf X=(X,\I)$ has $\Phi_{\mathbf X}(n)\le \hbar(n)\le\varphi(n+1)\le n!$ for every $n\ge 2$.
\end{proposition}

\begin{proof} The inequalities $\Phi_{\mathbf X}(n)\le \hbar(n)\le\varphi(n+1)$ follow from Theorems~\ref{t2.4} and \ref{t2.5n}. The inequality $\varphi(n+1)\le n!$ will be proved by induction.
It holds for $n=2$ as $\varphi(3)=\sum_{i=0}^{3-1-1}1^i=2=2!$. Assume that for some $n\ge 1$ we have proved that $\varphi(n)\le (n-1)!$. Observe that for every $0<k<n$
$$
\sum_{i=0}^{n-k}k^i=\sum_{i=0}^{n-1-k}k^i+k^{n-k}\le \varphi(n)+\frac{k^{n-k}-1}{k-1}(k-1)+1\le \varphi(n)+\varphi(n)(k-1)+1=\varphi(n)k+1\le \varphi(n)(k+1),
$$
which implies
$$\varphi(n+1)=\max_{0<k\le n}\sum_{i=0}^{n-k}k^i=\max_{0<k<n}\sum_{i=0}^{n-k}k^i\le\max_{0<k< n}\varphi(n)(k+1)=\varphi(n)\cdot n\le (n-1)!\cdot n=n!.$$
\end{proof}

The definition of the numbers $\hbar(n)$ is algorithmic and can be calculated by computer.
However the complexity of calculation grows very quickly. So, the exact values of the numbers $\hbar(n)$ are known only $n\le 7$. For $n=8$ by long computer calculations we have merely found an upper bound on $\hbar(8)$. In particular, finding the upper bound $\hbar(8)\le 136$ required a year of continuous calculations on a laptop computer. The values of the sequences $\varphi(n)$, $1+\lfloor\phi(n)\rfloor$, $\hbar(n)$,  $\varphi(n+1)$, $n!$, $u(n)$, and $2^{2^{n-1}-1}$ for $n\le 8$ are presented in the following table:

$$
\begin{array}{|c|rrrrrrr}\hline
n  & 2 & 3 & 4 & 5 & 6 & 7 & 8 \\ \hline
\varphi(n)&1 &2 &3&7&15&40&121  \\
1+\lfloor\phi(n)\rfloor&2 &3&4&8&17&42&122  \\ \hline
\hbar(n)& 2 & 3 & 5 & 9 & 19&47&{\le}136\\ \hline
\varphi(n+1)&2 &3&7&15&40&121&364  \\ \hline
n!& 2 & 6 &24& 120 &720&4320&30240 \\ \hline
u(n) &2&6& 42& 1806&   3263442&     10650056950806&---\\
2^{2^{n-1}-1}&2&8&128&32768&2147483648&9223372036854775808&2^{127}\\ \hline
\end{array}
$$
\smallskip

This table shows that the upper bound given by Theorem~\ref{t2.4} is much better than those from Theorems~\ref{t2.1}---\ref{t2.3}. Since $\hbar(n)=n$ for $n\le 3$, Theorem~\ref{t2.4} implies a positive answer to Problem~\ref{prob3} for $n\le 3$.

\begin{corollary} For each partition $X=A_1\cup\dots\cup A_n$ of an ideal $G$-space $\mathbf X=(X,\I)$ into $n\le 3$ pieces some cell $A_i$ of the partition has $\cov(\Delta_\I(A_i))\le n$. Consequently, $\Phi_{\mathbf X}(n)\le n$ for $n\le 3$.
\end{corollary}

\section{Densities and submeasures on $G$-spaces}\label{s3}

The other approach to solution of Problems~\ref{prob1}--\ref{prob3} exploits various densities and submeasures on $G$-spaces. Partial solutions of Problems~\ref{prob1}--\ref{prob3} obtained by this method
 are surveyed in Sections~\ref{s3}--\ref{s9}. In this section we recall the necessary definitions related to densities and submeasures.

Let $X$ be a $G$-space and $\PP(X)$ be the Boolean algebra of all subsets of $X$.
A function $\mu:\PP(X)\to [0,1]$ is called
\begin{itemize}
\item {\em $G$-invariant} if $\mu(gA)=\mu(A)$ for any $A\subset X$ and $g\in G$;
\item {\em monotone} if $\mu(A)\le\mu(B)$ for any subsets $A\subset B\subset X$;
\item {\em subadditive} if $\mu(A\cup B)\le\mu(A)+\mu(B)$ for any subsets $A,B\subset X$;
\item {\em additive} if $\mu(A\cup B)=\mu(A)+\mu(B)$ for any disjoint subsets $A,B\subset X$;
\item a {\em density} if $\mu$ is monotone, $\mu(\emptyset)=0$ and $\mu(X)=1$;
\item a {\em submeasure} if $\mu$ is a subadditive density;
\item a {\em measure} if $\mu$ is an additive density.
\end{itemize}
So, all our measures are finitely additive probability measures defined on the Boolean algebra $\PP(X)$ of all subsets of $X$.

The space of all densities on $X$ will be denoted by $D(X)$ and will be considered as a closed convex subspace of the Tychonoff cube $[0,1]^{\PP(X)}$. The space $D(X)$ contains a closed convex subset $P(X)$ consisting of measures.

A measure $\mu$ on $X$ is called {\em finitely supported} if $\mu(F)=1$ for some finite subset $F\subset X$. In this case $\mu$ can be written as a convex combination $\mu=\sum_{i=1}^n\alpha_i\delta_{x_i}$ of Dirac measures. Let us recall that the {\em Dirac measure} $\delta_x$ supported at a point $x\in X$ is the $\{0,1\}$-valued measure assigning to each subset $A\subset X$ the number
$$\delta_x(A)=\begin{cases}0&\mbox{if $x\notin A$},\\
1&\mbox{if $x\in A$}.
\end{cases}
$$
The family of all finitely supported measures on $X$ will be denoted by $P_\w(X)$. It is a convex dense subset in the space $P(X)$ of all measures on $X$.

A finitely supported measure $\mu\in P_\w(X)$ is called a {\em uniformly distributed measure} if $\mu=\frac1{|F|}\sum_{x\in F}\delta_x$ for some non-empty finite subset $F\subset X$ (which coincides with the support of the measure $\mu$). The set of uniformly distributed measures will be denoted by $P_u(X)$.

For any finitely supported measures $\mu=\sum_i\alpha_i\delta_{a_i}$ and $\nu=\sum_j\beta_j\delta_{b_j}$ on a group $G$ we can define their convolution by the formula
$$\mu*\nu=\sum_{i,j}\alpha_i\beta_j\delta_{a_ib_j}.$$
More generally, the convolution $\mu*\nu$ can be well-defined for any measure $\mu\in P(G)$ on a group $G$ and a density $\nu\in D(X)$ on a $G$-space $X$:
$$\mu*\nu(A)=\int_G\nu(x^{-1}A)d\mu(x)\mbox{ for any set $A\subset X$}.$$
It can be shown that the convolution operation  $*:P(G)\times D(X)\to D(X)$ is {\em right continuous} in the sense that for every density $\nu\in D(X)$ the right shift $\rho_\nu:P(G)\to D(X)$, $\rho_\nu:\mu\mapsto \mu*\nu$, is continuous.
By a standard argument (see e.g. \cite[4.4]{HS}), it can be shown that the operation of convolution is associative in the sense that
$$(\mu_1*\mu_2)*\mu_3=\mu_1*(\mu_2*\mu_3)$$for any measures $\mu_1,\mu_2\in P(G)$ and a density $\mu_3\in D(X)$.
The right-continuity of the convolution operation implies that for every density $\nu\in D(X)$ its $P(G)$-orbit $P(G)*\nu=\{\mu*\nu:\mu\in P(G)\}$ is a closed convex set in $D(X)$, which coincides with the closed convex hull of the $G$-orbit $G\nu=\{\delta_g*\nu:g\in G\}$ of $\nu$.

A density $\nu\in D(X)$ will be called {\em minimal} if each density $\mu\in P(G)*\nu$ has $P(G)$-orbit $P(G)*\mu=P(G)*\nu$. Zorn's Lemma and the compactness of $P(G)$-orbits implies that for each density $\nu\in D(X)$ its $P(G)$-orbit $P(G)*\nu$ contains a minimal density.

Let $D_{\min}(X)$ be the set of all minimal densities on $X$ and $P_{\min}(X)=P(X)\cap D_{\min}(X)$ be the set of all minimal measures on $X$. Observe that the set $P_{\min}(X)$ is not empty and contains the set $P_G(X)$ of all $G$-invariant measures (which can be empty). A $G$-space $X$ is called {\em amenable} if $P_G(X)\ne \emptyset$, i.e., $X$ admits a $G$-invariant measure $\mu:\PP(X)\to[0,1]$.
It can be shown that a $G$-space $X$ is amenable if it satisfies the {\em F\o lner condition}:
for every $\e>0$ and every finite set $F\subset G$ there is a finite set $E\subset X$ such that $|FE|<(1+\e)|E|$. It is well-known \cite{Pat} that a group $G$ is amenable if and only if it satisfies the  F\o lner condition.

Each group $G$ will be considered as a $G$-space endowed with the left action $\alpha:G\times G\to G$, $\alpha:(g,x)\mapsto gx$, of $G$ on itself. In this case the space $P(G)$ endowed with the operation of convolution is a compact right-topological semigroup. $G$-Invariant densities or Boolean ideals on $G$ will be called {\em left-invariant}.

For a density $\mu:\PP(X)\to[0,1]$ on a set $X$ its {\em subadditivization} $\widehat\mu:\PP(X)\to[0,1]$ is defined by the formula
$$\widehat\mu(A)=\sup_{B\subset X}\big(\mu(A\cup B)-\mu(B)\big).$$
The subadditivization $\widehat \mu$ is a submeasure such that $\mu\le\widehat\mu$. A density $\mu$ is subadditive if and only if it $\mu=\widehat\mu$.

For a density $\mu:\PP(X)\to[0,1]$ on a set $X$ by $[\mu{=}0]$ we shall denote the family $\{A\subset X:\mu(A)=0\}$. The family $[\mu{=}0]$ is a non-trivial Boolean ideal on $X$ if $\mu$ is subadditive. The inequality $\mu\le\widehat\mu$ implies $[\widehat\mu{=}0]\subset[\mu{=}0]$ for any density $\mu$ on $X$.

In this paper we shall meet many examples of so-called {\em extremal densities}. Those are densities obtained by applying infima and suprema to convolutions of measures over certain families of measures on groups or $G$-spaces. The simplest examples of extremal densities are the densities
$\mathsf i_1:\PP(X)\to\{0,1\}$ and $\mathsf s_1:\PP(X)\to[0,1]$ defined on each set $X$ by
$$\mathsf i_1(A)=\inf_{\mu_1\in P_\w(X)}\mu_1(A)\mbox{ \ \ \ and \ \ \ }\mathsf s_1(A)=\sup_{\mu_1\in P_\w(X)}\mu_1(A).$$
It is clear that
$$
\mathsf i_1(A)=\begin{cases}0,&\mbox{if $A\ne X$,}\\
1&\mbox{if $A=X$,}\end{cases}
\mbox{ \ \ \ and \ \ \ }
\mathsf s_1(A)=\begin{cases}0,&\mbox{if $A=\emptyset$,}\\
1&\mbox{if $A\ne\emptyset$,}\end{cases}
$$
which implies that $\mathsf i_1$ and $\mathsf s_1$ are the smallest and largest densities on $X$, respectively. The density $\mathsf s_1$ is subadditive while $\mathsf i_1$ is not (for a set $X$ containing more than one point). More complicated extremal densities will appear in Sections~\ref{s5}--\ref{s9}.

Another important example of an extremal density is the {\em upper Banach density}
$$d^*:\PP(X)\to[0,1],\;\;d^*:A\mapsto \sup_{\mu\in P_{\min}(X)}\mu(A)$$ defined on each  $G$-space $X$.
It is clear that the upper Banach density $d^*$ is subadditive and hence is a submeasure on $X$.

\section{On partitions of $G$-spaces endowed with a $G$-invariant measure}\label{s4}

In fact, Problems~\ref{prob1}--\ref{prob3} have trivial affirmative answer for amenable $G$-spaces (cf. Theorem 12.8 \cite{PB}).

\begin{theorem}\label{t4.1} Let $(X,\I)$ be an ideal $G$-space endowed with a $G$-invariant measure $\mu:\PP(X)\to[0,1]$ such that $\I\subset[\mu{=}0]$. Each subset $A\subset X$ with $\mu(A)>0$ has $\cov(\Delta_\I(A))\le1/\mu(A).$
\end{theorem}

\begin{proof} Choose a maximal subset $F\subset G$ such that $\mu(xA\cap yA)=0$ for any distinct points $x,y\in F$. The additivity and $G$-invariance of the measure $\mu$ implies that $|F|\le1/\mu(A)$. By the maximality of $F$, for every $x\in G$ there is $y\in F$ such that $\mu(xA\cap yA)>0$, which implies $yA\cap xA\notin\I$ and $A\cap y^{-1}xA\notin \I$. Then $y^{-1}x\in\Delta_\I(A)$ and $x\in y\cdot \Delta_\I(A)$. So, $X=F\cdot\Delta_\I(A)$ and $\cov(\Delta_\I(A))\le |F|\le 1/\mu(A)$.
\end{proof}

\begin{corollary}\label{c4.2} Let $(X,\I)$ be an ideal $G$-space admitting a $G$-invariant measure $\mu:\PP(X)\to[0,1]$ such that $\I\subset[\mu{=}0]$. For each partition $X=A_1\cup \dots \cup A_n$ of $X$ some cell $A_i$ of the partition has $\cov(\Delta_\I(A_i))\le n$. This implies  $\Phi_{(X,\I)}(n)\le n$ for all $n\in\IN$.
\end{corollary}

\begin{proof} The subadditivity of the measure $\mu$ guarantees that some cell $A_i$ of the partition has measure $\mu(A_i)\ge1/n$. Then $\cov(\Delta_\I(A_i))\le 1/\mu(A_i)\le n$ according to Theorem~\ref{t4.1}.
\end{proof}

The following theorem resolves Problem~\ref{prob2} for $G$-spaces with amenable acting group $G$.

\begin{theorem}\label{t4.3} For each partition $G=A_1\cup\dots\cup A_n$ of an ideal $G$-space $(X,\I)$ endowed with an action of an amenable group $G$, some cell $A_i$ of the partition has $\cov(\Delta_\I(A_i))\le n$, which implies $\Phi_{(G,\I)}(n)\le n$ for all $n\in\IN$.
\end{theorem}

\begin{proof} Using the idea of the proof of Theorem 3.1 of \cite{BL}, we shall construct a $G$-invariant measure $\mu:\PP(X)\to[0,1]$ of $X$ whose null ideal $[\mu{=}0]$ contains the  $G$-invariant ideal $\I$.

Let $[G]^{<\w}$ be the Boolean ideal of all finite subsets of the amenable group $G$. Consider the set $\mathcal D=([G]^{<\w}\setminus\{\emptyset\})\times \IN\times \I$ endowed with the partial order $(F,n,A)\le (E,m,B)$ iff $F\subset E$, $n\le m$, and $A\subset B$. To each triple $d=(F,n,A)$ assign a finitely supported measure $\mu_d$ on $X$ as follows. Using the F\o lner condition, find a finite set $E\subset [G]^{<\w}$ such that $|FE|<(1+\frac1n)|E|$. Since $\I$ is a non-trivial $G$-invariant ideal on $X$, the set $E^{-1}A\in\I$ does not coincide with $X$ and hence we can find a point $x_d\in X\setminus E^{-1}A$.
Then $Ex_d\subset X\setminus A$ and hence $\mu_d(A)=0$ for the finitely supported measure
$\mu_d=\frac1{|E|}\sum_{g\in E}\delta_{gx_d}$ on $X$.

By the compactness of the Tychonoff cube $[0,1]^{\PP(X)}$ the net $(\mu_d)_{d\in\mathcal D}$ has a limit point, which is a measure $\mu:\PP(X)\to[0,1]$ such that for every neighborhood $O(\mu)\subset [0,1]^{\PP(X)}$ and every $d_0\in \mathcal D$ there is $d\ge d_0$ in $\mathcal D$ such that $\mu_d\in O(\mu)$. Repeating the argument of the proof of Theorem 3.1 \cite{BL} it can be shown that $\mu$ is a $G$-invariant measure on $X$ such that $\mu(A)=0$ for every $A\in\mathcal I$.

So, it is legal to apply Corollary~\ref{c4.2} and conclude that for any partition $X=A_1\cup\dots\cup A_n$ of $X$ some cell $A_i$ of the partition has $\cov(\Delta_\I(A_i))\le n$.
\end{proof}

So, Problems~\ref{prob2}, \ref{prob3} remains open only for $G$-spaces with non-amenable acting group $G$.

\section{The extremal density $\is_{12}$}\label{s5}

In this section we consider the extremal density $\is_{12}$, which is defined on each $G$-space $X$ by the formula
$$\is_{12}(A)=\inf_{\mu_1\in P_\w(G)}\sup_{\mu_2\in P_\w(X)}\mu_1*\mu_2(A)=\inf_{\mu\in P_\w(G)}\sup_{x\in X}\mu*\delta_x(A)$$
for $A\subset X$. It can be shown that the density $\is_{12}$ is $G$-invariant. In case of groups this density (denoted by $a$) was introduced in \cite{Sol} and later studied in \cite{Ban}.


The density $\is_{12}(A)$ can be used to give an upper bound for the packing index of a set $A$ in $G$. For a subset $A\subset X$ of an ideal $G$-space $(X,\I\/)$ its packing index $\Ipack(A)$ is defined by $$\Ipack(A)=\sup\{|E|:E\subset G,\;\;xA\cap yA\in\I\mbox{ for any distinct points $x,y\in E$}\}.$$
If the ideal $\I=\{\emptyset\}$, the we shall write $\pack(A)$ instead of $\pack_{\{\emptyset\}}(A)$.
Packing indices were introduced and studied in \cite{BLR}, \cite{BL11}.
The packing index $\Ipack(A)$ upper bounds the covering number $\cov(\Delta_\I(A))$.

\begin{proposition}\label{p5.2n} For any subset $A$ of an ideal $G$-space $(X,\I)$ we get $\cov_\I(A)\le\Ipack(A).$
\end{proposition}

\begin{proof} Using Zorn's Lemma, choose a maximal subset $F\subset G$ such that $xA\cap yA\in\I$ for any distinct points $x,y\in F$. By the maximality of $F$, for any $x\in G$ there is $y\in F$ such that $yA\cap xA\notin\I$. By the $G$-invariance of the ideal $\I$, $A\cap y^{-1}xA\notin\I$ and hence $y^{-1}x\in\Delta_\I(A)$. Then $x\in y\,\Delta_\I(A)\subset F\cdot\Delta_\I(A)$ and hence
$$\cov(\Delta_\I(A))\le|F|\le\Ipack(A).$$
\end{proof}

Applications of the extremal density $\is_{12}$ to Problems~\ref{prob1}--\ref{prob3} are based on the following fact.

\begin{proposition}\label{p5.2} If a subset $A$ of a $G$-space $X$ has positive density $\is_{12}(A)>0$, then  $$\cov(\Delta_\I(A))\le\Ipack(A)\le 1/\is_{12}(A)$$
for any $G$-invariant Boolean ideal $\I\subset[\wis_{12}{=}0]$.
\end{proposition}

\begin{proof} The inequality $\cov(\Delta_\I(A))\le\Ipack(A)$ was proved in Proposition~\ref{p5.2n}. It remains to prove that $\Ipack(A)\le 1/\is_{12}(A)$.
Assuming conversely that $\Ipack(A)>1/\is_{12}(A)$, we can find a finite subset $F\subset G$ such that $|F|>1/\is_{12}(A)$ and $xA\cap yA\in\I$ for any distinct points $x,y\in F$. It follows that the set $Z=\bigcup\{xA\cap yA:x,y\in F,\;x\ne y\}$ belongs to the ideal $\I$ and so does the set $F^{-1}Z$. Consequently, $\wis_{12}(F^{-1}Z)=0$ and the set $A'=A\setminus F^{-1}Z$ has density
$\is_{12}(A')=\is_{12}(A)$ according to the definition of the submeasure $\wis_{12}$.
The definition of the set $Z$ implies that the indexed family $(xA')_{x\in F}$ is disjoint. We claim that $|F^{-1}z\cap A'|\le 1$ for every point $z\in X$. Assuming conversely that for some $z\in X$ the set $F^{-1}z$ contains two distinct points $a,b\in A'$, we conclude that $a=x^{-1}z$ and $b=y^{-1}z$ for two distinct points $x,y\in F$, which implies that $xA'\cap yA'\ni z$ is not empty. But this contradicts the disjointness of the family $(xA')_{x\in F}$.
So, $|F^{-1}z\cap A'|\le 1$ and hence for the uniformly distributed measure $\mu_1=\frac1{|F|}\sum_{g\in F}\delta_{g^{-1}}$ we get
$$
\begin{aligned}
\is_{12}(A)&=\is_{12}(A')\le\sup_{\mu_2\in P_\w(X)}\mu_1*\mu_2(A')=\sup_{x\in X}\mu_1*\delta_x(A')=\\
&=\sup_{x\in X}\frac1{|F|}\sum_{g\in F}\delta_{g^{-1}}*\delta_z(A')=\sup_{x\in X}\frac{|F^{-1}z\cap A'|}{|F|}\le\frac1{|F|}<\is_{12}(A),
\end{aligned}
$$
which is a desired contradiction proving that $\Ipack(A)\le1/\is_{12}(A)$.
\end{proof}

\begin{corollary} If for a $G$-space $X$ the extremal density $\is_{12}$ is subadditive, then  any partition $X=A_1\cup\dots\cup A_n$ of $X$ some cell $A_i$ of the partition has $\cov(\Delta_\I(A))\le n$ for any left-invariant Boolean ideal $\I\subset [\is_{12}{=}0]$.
\end{corollary}

\begin{proof} The subadditivity of the density $\is_{12}$ on implies that $\wis_{12}=\is_{12}$ and $\I\subset[\is_{12}{=}0]=[\wis_{12}{=}0]$. Also the subadditivity of $\is_{12}$ guarantees that some cell $A_i$ of the partition has density $\is_{12}(A_i)\ge1/n$. Then $\cov(\Delta_\I(A_i))\le 1/\is_{12}(A_i)\le n$ according to Proposition~\ref{p5.2}.
\end{proof}

The following fact was proved in Theorem 3.9 of \cite{Ban2}.

\begin{proposition}\label{p5.1} For any $G$-space $X$ with amenable acting group $G$ the extremal density $\is_{12}$ coincides with the upper Banach density $d^*$ and hence is subadditive.
\end{proposition}

\section{The extremal density $\us_{12}$}

In this section we consider a uniform variation of the extremal density $\is_{12}$, denoted by $\us_{12}$. On each $G$-space $X$ the extremal density $\us_{12}:\PP(X)\to[0,1]$ is defined by
$$\us_{12}(A)=\inf_{\mu_1\in P_u(G)}\sup_{\mu_2\in P_\w(X)}\mu_1*\mu_2(A)\mbox{ \ \ for $A\subset X$}.
$$Here $P_u(G)$ stands for the set of all uniformly distributed measures on $G$.
It can be shown that on a group $G$ the density $\us_{12}$ can be equivalently defined as
$$\us_{12}(A)=\inf_{\emptyset{\ne}F\in[G]^{<\w}}\sup_{x\in X}\frac{|Fx\cap A|}{|F|}.$$
where $[G]^{<\w}$ is the Boolean ideal consisting of all finite subsets of $G$.
On groups the density $\us_{12}$ (denoted by $u$) was introduced by Solecki in \cite{Sol} and studied in more details in \cite{Sol} and \cite{Ban}.

It can be shown that the density $\us_{12}$ is $G$-invariant on each $G$-space $X$ and $\is_{12}\le\us_{12}$. Moreover a subset $A\subset X$ has $\is_{12}(A)=1$ if and only if $\us_{12}(A)=1$ if and only if $A$ is {\em thick} in $X$ in the sense that for every finite subset $F\subset G$ there is a point $x\in X$ with $Fx\subset A$.

On amenable groups the densities $\us_{12}$ and $\is_{12}$ coincide. This was shown by Solecki in \cite{Sol}:

\begin{proposition}[Solecki] For any amenable group $G$ the densities $\us_{12}$ and $\is_{12}$ coincide and are subadditive. If a group $G$ contains a non-Abelian free subgroup, then for every $\e>0$ there is a set $A\subset G$ with $\is_{12}(A)<\e$ and $\us_{12}(A)>1-\e$.
\end{proposition}

In general, the density $\us_{12}$ is not subadditive (as well as the density $\is_{12}$):

\begin{example}\label{ex5.1} The free group with two generators can be written as the union $F_2=A\cup B$ of two sets with $\us_{12}(A)=\us_{12}(B)=0$.
\end{example}

\begin{proof} Let $a,b$ be the generators of the free group $F_2$. The elements of the group $F_2$ can be identified with irreducible words in the alphabet $\{a,b,a^{-1},b^{-1}\}$. Let $A$ be the set of irreducible words that start with $a$ or $a^{-1}$ and $B=F_2\setminus A$. It can be shown that $F_2=A\cup B$ is a required partition with $\us_{12}(A)=\us_{12}(B)=0$. For details, see Example 3.2 in \cite{Ban}.
\end{proof}

The extremal density $\us_{12}$ can be adjusted to a subadditive density $\wus_{12}:\PP(X)\to[0,1]$ defined by
$\wus_{12}(A)=\sup_{B\subset X}(\us_{12}(A\cup B)-\us_{12}(B))$ for $A\subset X$.

For our purposes, the density $\us_{12}$ will be helpful because of the following its property, which is a bit stronger than Proposition~\ref{p5.2} and can be proved by analogy:

\begin{proposition}\label{p6.3} If a subset $A$ of a $G$-space $G$ has positive density $\us_{12}(A)>0$, then $$\cov(\Delta_\I(A))\le \Ipack(A)\le 1/\us_{12}(A)$$ for any $G$-invariant Boolean ideal $\I\subset [\wus_{12}{=}0]$.
\end{proposition}

\section{The extremal submeasure $\sis_{123}$}\label{s9}

In this section we shall present applications of the $G$-invariant submeasure $\sis_{123}$ defined on each $G$-space $X$ by the formula
$$\sis_{123}(A)=\sup_{\mu_1\in P_\w(G)}\inf_{\mu_2\in P_\w(G)}\sup_{\mu_3\in P_\w(X)}\mu_1*\mu_2*\mu_3(A)=\sup_{\mu\in P_\w(G)}\inf_{\nu\in P_\w(G)}\sup_{x\in X}\mu*\nu*\delta_x(A)\mbox{ \ \ for \ \ $A\subset X$}.$$

\begin{proposition}\label{p7.1} On each $G$-space $X$ the density $\sis_{123}:\PP(X)\to[0,1]$ is subadditive.
\end{proposition}

\begin{proof} It suffices to check that $\sis_{123}(A\cup B)\le\sis_{123}(A)+\sis_{123}(B)+2\e$ for every subsets $A,B\subset X$ and real number $\e>0$. This will follow as soon as for any measure $\mu_1\in P_\w(G)$ we find a measure $\mu_2\in P_\w(G)$ such that $\sup_{\mu_3\in P_\w(X)}\mu_1*\mu_2*\mu_3(A\cup B))<\sis_{123}(A)+\sis_{123}(B)+2\e$.

By the definition of $\sis_{123}(A)$, for the measure $\mu_1$ there is a measure $\nu_2\in P_\w(G)$ such that
$$\sup_{\nu_3\in P_\w(X)}\mu_1*\nu_2*\mu_3(A)<\sis_{123}(A)+\e.$$ By the definition of $\sis_{123}(B)$ for the measure $\eta_1=\mu_1*\nu_2$ there is a measure $\eta_2\in P_\w(G)$ such that $$\sup_{\eta_3\in P_\w(X)}\eta_1*\eta_2*\eta_3(B)<\sis_{123}(B)+\e.$$ We claim that the measure $\mu_2=\nu_2*\eta_2$ has the required property. Indeed, for every measure $\mu_3\in P_\w(X)$ we get
$$\mu_1*\mu_2*\mu_3(A\cup B)\le \mu_1*\nu_2*(\eta_2*\mu_3)(A)+(\mu_1*\nu_2)*\eta_2*\mu_3(B)<\sis_{123}(A)+\e+\sis_{123}(B)+\e.$$
\end{proof}

The submeasure $\sis_{123}$ yields an upper bound on the extremal density $\is_{12}$.
The following fact was proved in \cite{Ban2}.

\begin{proposition}\label{p10.1} For any $G$-space we get $\is_{12}\le\wis_{12}\le \sis_{123}\le d^*$. Moreover, if the acting group $G$ is amenable, then
 $\us_{12}=\is_{12}=\wis_{12}=\sis_{123}=d^*$.
\end{proposition}

\begin{proof} For convenience of the reader we present a proof of the inequality $\wis_{12}\le\sis_{123}$. It suffices to check that
$$\is_{12}(A\cup B)<\is_{12}(A)+\sis_{123}(B)+2\e$$for every subsets $A,B\subset X$ and every $\e>0$.
By the definition of $\is_{12}(A)$, there is a measure $\mu_1\in P_\w(G)$ such that $\sup_{\mu_2\in P_\w(X)}\mu_1*\mu_2(A)<\is_{12}(A)+\e$.
By the definition of $\sis_{123}(B)$, for the measure $\mu_1$ there is a measure $\mu_2\in P_\w(G)$ such that $\sup_{\mu_3\in P_\w(X)}\mu_1*\mu_2*\mu_3(B)\le \sis_{123}(B)+2$.  Then for the measure $\nu_1=\mu_1*\mu_2\in P_\w(G)$ we get
$$\is_{12}(A\cup B)\le\sup_{\nu_2\in P_\w(X)}\nu_1*\nu_2(A\cup B)\le \sup_{\nu_2\in P_\w(X)}(\mu_1*\mu_2*\nu_2(A)+\mu_1*\mu_2*\nu_2(B))<\is_{12}(A)+\sis_{123}(B)+2\e.$$
\end{proof}

\begin{proposition}\label{p10.3} For any $G$-space $X$ with finite acting group $G$ we get
$$\is_{12}(A)=\sis_{123}(A)=d^*(A)=\sup_{x\in X}\frac{|A\cap Gx|}{|Gx|}$$
for every set $A\subset X$.
\end{proposition}

\begin{proof} Denote by $\lambda=\frac1{|G|}\sum_{x\in G}\delta_x$ the Haar measure on the group $G$ and observe that for every $A\subset X$ we get
$$\sis_{123}(A)\le\sup_{\mu_1\in P_\w(G)}\sup_{x\in X}\mu_1*\lambda*\delta_x(A)=\sup_{x\in X}\lambda*\delta_x(A)=\sup_{x\in X}\frac{|A\cap Gx|}{|Gx|}.$$
On the other hand,
$$\is_{12}(A)=\inf_{\mu_1\in P_\w(X)}\sup_{\mu_2\in P_\w(X)}\mu_1*\mu_2(A)\ge\inf_{\mu_1\in P_\w(G)}\sup_{x\in X}\mu_1*\lambda*\delta_x=\sup_{x\in X}\lambda*\delta_x(A)=\sup_{x\in X}\frac{|A\cap Gx|}{|Gx|}.$$
\end{proof}

A subset $A$ of a group $G$ is called {\em conjugacy-invariant} if $xAx^{-1}=A$ for every $x\in G$.

\begin{proposition}\label{p7.5} Each conjugacy-invariant subset $A$ of a group $G$ has density $\is_{12}(A)=\sis_{123}(A)$.
\end{proposition}

\begin{proof} The inequality $\is_{12}(A)\le\sis_{123}(A)$ follows from Proposition~\ref{p10.1}.
To prove that $\sis_{123}(A)\le \is_{12}(A)$, fix any $\e>0$ and find a measure $\nu\in P_\w(G)$ such that $\sup_{\eta\in P_\w(G)}\nu*\eta(A)<\is_{12}(A)+\e$. Given any measure $\mu_1=\sum_{i}\alpha_i\delta_{a_i}\in P_\w(G)$ put $\mu_2=\nu$ and observe that for every $\mu_3\in P_\w(G)$ we get
$$
\begin{aligned}
\mu_1*\mu_2*\mu_3(A)&=\sum_i\alpha_i\, \delta_{a_i}*\nu*\mu_3(A)=\sum_i\alpha_i\, \nu*\mu_3(a_i^{-1}A)=\sum_i\alpha_i\,  \nu*\mu_3(Aa_i^{-1})=\\
&=\sum_i\alpha_i\,  \nu*\mu_3*\delta_{a_i}(A)\le \sum_i\alpha_i\sup_{\eta\in P_\w(G)}\nu*\eta(A)<\sum_i\alpha_i(\is_{12}(A)+\e)=\is_{12}(A)+\e.
\end{aligned}
$$
This implies that $\sis_{123}(A)\le\is_{12}(A)+\e$ for every $\e>0$ and hence $\sis_{123}(A)\le\is_{12}(A)$.
\end{proof}

For partitions of groups into conjugacy-invariant sets Propositions~\ref{p7.1}, \ref{p7.5} and \ref{p5.2} imply the following partial answer to Problem~\ref{prob1}.

\begin{corollary} For any partition $G=A_1\cup\dots\cup A_n$ of a group $G$ into conjugacy-invariant sets  some cell $A_i$ of the partition has $\cov(\Delta_\I(A_i))\le n$ for any left-invariant Boolean ideal $\I\subset[\wis_{12}{=}0]$.
\end{corollary}

Applications of the submeasure $\sis_{123}$ will be based on the following theorem.

\begin{theorem}\label{t9.2} If a subset $A$ of a $G$-space $X$ has positive submeasure
$\sis_{123}(A)>0$, then for some finite set $E\subset G$ the set $\Delta_\I(A)^{\wr E}=\bigcup_{x\in E}x^{-1}\Delta_\I(A)x$ has $$\cov(\Delta_\I(A)^{\wr E})\le 1/\sis_{123}(A)$$for any $G$-invariant ideal $\I\subset[\sis_{123}{=}0]$.
\end{theorem}

\begin{proof} Fix $\e>0$ so small that each integer number $n\le\frac1{\sis_{123}(A)-2\e}$ does not exceed $\frac1{\sis_{123}(A)}$. By the definition of the submeasure $\sis_{123}(A)$, there is a measure $\mu_1\in P_\w(G)$ such that $$\inf_{\mu_2\in P_\w(G)}\sup_{\mu_3\in P_\w(X)}\mu_1*\mu_2*\mu_3(A)>\sis_{123}(A)-\e.$$ Write $\mu_1$ as a convex combination $\mu_1=\sum_{i=1}^n\alpha_i\delta_{a_i}$ and put $E=\{a_1,\dots,a_n\}$.

Using Zorn's Lemma, choose a maximal subset $M\subset G$ such that for every $a\in E$ and distinct $x,y\in M$ we get $xa^{-1}A\cap ya^{-1}A\in\I$. By the maximality of $M$, for every point $g\in G$ there are points $x\in M$ and $a\in E$ such that $ga^{-1}A\cap xa^{-1}A\notin\I$ and hence $ax^{-1}ga^{-1}\in\Delta_\I(A)$ and $g\in xa^{-1}\Delta_\I(A)a\subset M\cdot\Delta_\I(A)^{\wr E}$.
So, $G=M\cdot\Delta_\I(A)^{\wr E}$ and hence $\cov(\Delta_\I(A)^{\wr E})\le |M|$. To complete the proof, it remains to check that the set $M$ has cardinality $|M|\le 1/(\sis_{123}(A)-2\e)$.

Assuming the opposite, we could find a finite subset $F\subset M$ of cardinality $|F|>1/(\sis_{123}(A)-2\e)$. The choice of the set $M\supset F$ guarantees that the set
$$B=\bigcup_{i=1}^n\{xa_i^{-1}A\cap ya^{-1}A:x,y\in F,\;x\ne y\}$$ belongs to the ideal $\I$ and hence $B\in\I\subset[\iss_{123}{=}0]$. Put $A'=A\setminus B$ and observe that for every $a\in E$ the indexed family $(ga^{-1}A')_{g\in F}$ is disjoint.
Consider the uniformly distributed measure $\mu_F=\frac1{|F|}\sum_{g\in F}\delta_{g^{-1}}$ on $G$.
Since $\sis_{123}(B)=0$, for the measure $\nu_1=\mu_1*\mu_F\in P_\w(G)$
there is a measure $\nu_2=\sum_j\beta_j\delta_{b_j}\in P_\w(G)$ such that $\sup_{\nu_3(X)}\nu_1*\nu_2*\nu_3(B)<\e$.

By the choice of the measure $\mu_1$ for the measure $\mu_2=\mu_F*\nu_2\in P_\w(G)$
there is a measure $\mu_3\in P_\w(X)$ such that $\mu_1*\mu_2*\mu_3(A)>\sis_{123}(A)-\e$. The measure $\mu_3$ can be assumed to be a Dirac measure $\mu_3=\delta_{x}$ at some point $x\in X$. Then $\mu_1*\mu_2*\delta_{x}(A)>\sis_{123}(A)-\e>1/|F|+\e$.

On the other hand, for every $i,j$ the disjointness of the families $(ga_i^{-1}A')_{g\in F}$ and  $(b_j^{-1}ga_i^{-1}A')_{g\in F}$ implies that $\sum_{g\in F}\delta_x(b_j^{-1}ga_{i}^{-1}A')\le 1$ and then
$$
\begin{aligned}
\mu_1*\mu_2*\delta_x(A')&=\mu_1*\mu_F*\nu_2*\delta_x(A')=\\
&=\sum_{i,j}\alpha_i\beta_j\sum_{g\in F}\frac1{|F|}\delta_{a_ig^{-1}b_jx}(A')
=\frac1{|F|}\sum_{i,j}\alpha_i\beta_j\sum_{g\in F}\delta_x(b_j^{-1}ga_i^{-1}A')\le\frac1{|F|}.
\end{aligned}
$$
Then
$$
\begin{aligned}
\sis_{123}(A)-\e&<\mu_1*\mu_2*\mu_3(A)\le\mu_1*\mu_2*\mu_3(A')+\mu_1*\mu_2*\mu_3(B)=\\
&<\mu_1*\mu_2*\delta_x(A')+\mu_1*\mu_F*\nu_2*\delta_x(B)<
\frac1{|F|}+\e<\sis_{123}(A)-\e
\end{aligned}$$
which is a desired contradiction.
 \end{proof}

The subadditivity of the density $\sis_{123}$ and Theorem~\ref{t9.2} imply the following corollary.

\begin{corollary}\label{c9.3} For any partition $X=A_1\cup\dots\cup A_n$ of an ideal $G$-space $(X,\I)$ with $\I\subset[\sis_{123}{=}0]$ some cell $A_i$ of the partition has $\cov\big(\Delta_\I(A)^{\conj[E]}\big)\le n$ for some finite set $E\subset G$.
\end{corollary}

Combining Theorem~\ref{t9.2} with Theorem~\ref{t2.1}, we get:

\begin{corollary}\label{c1new} If a subset $A$ of a $G$-space $X$ has positive submeasure
$\sis_{123}(A)>0$, then $$\cov(\Delta_\I(A)\cdot\Delta_\I(A))<\infty$$ for any $G$-invariant ideal $\I\subset [\sis_{123}{=}0]$ on $X$.
\end{corollary}

\begin{proof} By Theorem~\ref{t9.2}, there is a finite set $E\subset G$ such that $G=\bigcup_{x,y\in E}x\Delta_\I(A)y$. By Theorem~\ref{t2.1}, there are points $x,y\in E$ such that the set $x\Delta_\I(A)y\cdot(x\Delta_\I(A)y)^{-1}=x\Delta_\I(A)\cdot \Delta_\I(A)x^{-1}$ has finite covering number in $G$. Then the set $\Delta_\I(A)\cdot \Delta_\I(A)$ has finite covering number too.
\end{proof}

\section{Applications of the minimal and idempotent measures}\label{s8}

In this section we survey partial answers to Problem~\ref{prob2} obtained by Banakh and Fr\c aczyk \cite{BF} with help of minimal measures on $G$-spaces and quasi-invariant idempotent measures on groups. For any measure $\mu\in P(X)$ on a $G$-space $X$ let $\bar\mu:\PP(X)\to[0,1]$ be the submeasure on $X$ defined by $\bar\mu(A)=\sup_{x\in G}\mu(xA)$.

\begin{theorem}\label{t8.1} Let $(X,\I)$ be an ideal $G$-space and $\mu\in P_\I(X)$ be a minimal measure on $X$. If some subset $A\subset X$ has $\bar\mu(A)>0$, then the $\I$-difference set $\Delta_\I(A)$ has
\begin{enumerate}
\item $\cov\big(\Delta_\I(A)\cdot\Delta_\I(A)\big)\le1/\bar\mu(A)$;
\item $\covJ(\Delta_\I(A))\le1/\bar\mu(A)$ for some $G$-invariant ideal $\mathcal J\subset\{B\in\PP(G):\is_{12}(B^{-1})=0\}$ on $G$ with $\Delta_\I(A)\notin\mathcal J$.
\end{enumerate}
\end{theorem}

This theorem implies the following three results:

\begin{corollary}\label{c8.2}If a subset $A\subset X$ of a $G$-space $X$ has upper Banach density $d^*(A)>0$, then  $$\cov\big(\Delta_\I(A)\cdot\Delta_\I(A)\big)\le\frac1{d^*(A)}$$ for any $G$-invariant Boolean ideal $\I\subset[d^*=0]$ on $X$.
\end{corollary}

We recall that $d^*=\sup_{\mu\in P_{\min}(X)}\mu$.

\begin{corollary}\label{c8.3} For any partition $X=A_1\cup\dots\cup A_n$ of an ideal $G$-space $(X,\I)$ either
\begin{itemize}
\item $\cov(\Delta_\I(A_i))\le n=1/d^*(A)$ for all cells $A_i$ or else
\item some cell $A_i$ of the partition has $\cov(\Delta_\I(A_i)\cdot\Delta_\I(A_i))<n$ and $\covJ(\Delta_\I(A_i))<n$ for  some $G$-invariant ideal $\mathcal J\subset \{B\in\PP(G):\is_{12}(B^{-1})=0\}$ with $\Delta_\I(A_i)\not\in\mathcal J$.
\end{itemize}
\end{corollary}

\begin{corollary} For any partition $X=A_1\cup\dots\cup A_n$ of an ideal $G$-space $(X,\I)$ either $\cov(\Delta_\I(A_i))\le n$ for all cells $A_i$ or else  $\cov(\Delta_\I(A_i)\cdot\Delta_\I(A_i))<n$ for some cell $A_i$.
\end{corollary}

For partitions of groups we can prove a more precise result using quasi-invariant idempotent measures.
A measure $\mu\in P(G)$ on a group $G$ will be called
\begin{itemize}
\item {\em idempotent} if $\mu*\mu$;
\item {\em left quasi-invariant} (resp. {\em right quasi-invariant}) if there is a function $f:G\to[1,\infty)$ such that $f(x)\mu(xA)\le \mu(A)$ (resp. $f(x)\mu(Ax)\le\mu(A)$~) for any $A\subset G$ and $x\in G$;
\item {\em quasi-invariant} if there $\mu$ is left and right quasi-invariant.
\end{itemize}
A Boolean ideal $\I\subset \PP(G)$ on a group $G$ is called {\em invariant} if for every set $A\in\I$ and points $x,y\in G$ the shift $xAy\in\I$. The existence of quasi-invariant idempotent measures was established in \cite{BF}:

\begin{proposition} For any invariant ideal $\I$ on a countable group $G$ there is a quasi-invariant idempotent minimal measure $\mu\in P(G)$ such that $\mu(A)=0$ for all $A\in\I$.
\end{proposition}

Using quasi-invariant idempotent measures Banakh and Fr\c aczyk \cite{BF} proved the following result.

\begin{theorem}\label{t4} Let $\I$ be a $G$-invariant ideal on a group $G$ and $\mu\in P_\I(G)$ be a right quasi-invariant idempotent measure on $G$. If a subset $A\subset G$ has $\bar\mu(A)>0$, then its $\I$-difference set $\Delta_\I(A)$ has
\begin{enumerate}
\item $\cov(\Delta_\I(A)\cdot A)\le1/\bar\mu(A)$ and
\item $\covJ(\Delta_\I(A))\le1/\bar\mu(A)$ for some $G$-invariant Boolean ideal $\mathcal J\not\ni A^{-1}$ on $G$.
\end{enumerate}
\end{theorem}

This theorem implies the following partial answer to Problem~\ref{prob2}.

\begin{theorem}\label{t2} Let $G$ be a group and $\I$ be an invariant Boolean ideal on $G$ which does not contain some countable subset of $G$. For any partition $G=A_1\cup\dots\cup A_n$ of $G$ either
\begin{itemize}
\item $\cov(\Delta_\I(A_i))\le n$ for all cells $A_i$ or else
\item some cell $A_i$ of the partition has $\cov(\Delta_\I(A_i)\cdot A_i)<n$ and $\covJ(\Delta_\I(A_i))<n$ for some $G$-invariant ideal $\mathcal J\not\ni A_i^{-1}$ on $G$.
\end{itemize}
\end{theorem}

\begin{corollary}\label{c8.8} For any partition $G=A_1\cup\dots\cup A_n$ of a group $G$ either $\cov(A_iA_i)\le n$ for all cells $A_i$ or else $\cov(A_iA_i^{-1}\kern-1pt A_i)<n$ for some cell $A_i$ of the partition.
\end{corollary}

Taking into account that the ideal $\mathcal J$ appearing in Theorem~\ref{t2} is $G$-invariant but not necessarily invariant, we can ask the following question.

\begin{problem} Is it true that for any partition $G=A_1\cup\dots\cup A_n$ of a group $G$ some cell $A_i$ of the partition has $\covJ(A_iA_i^{-1})\le n$ for some invariant Boolean ideal $\mathcal J$ (for example, the ideal of small subsets) on $G$?
\end{problem}

Let us recall that a subset $A$ of a $G$-space $X$ is called {\em small} if $\cov(G\setminus FA)<\w$ for any finite subset $F\subset G$.

Corollary~\ref{c8.2} will help us to calculate the extremal densities of subgroups in groups. Below we assume that $1/\kappa=0$ for any infinite cardinal $\kappa$.

\begin{proposition}\label{t10.4} If $H$ is a subgroup of a group $G$, then
$\is_{12}(H)=\sis_{123}(H)=d^*(H)=1/{\cov(H)}$.
\end{proposition}

\begin{proof} Assume that the subgroup $H$ has infinite index $\cov(H)$ in $G$. We claim that $d^*(H)=0$. Assuming that $d^*(H)>0$ and applying Corollary~\ref{c8.2}, we conclude that
$\cov(H)=\cov(HH^{-1}HH^{-1})$ is finite and hence $H$ has finite index in $G$. This contradiction shows that $\is_{12}(H)\le\sis_{123}(H)\le d^*(H)=0=1/\cov(H)$.

Next, we assume that $H$ has finite index in $G$. Then the normal subgroup
$N=\bigcap_{x\in G}xHx^{-1}$ also has finite index in $G$. Consider the finite group $G/N$ and the quotient homomorphism $q:G\to G/N$. It follows that the subgroup $q(H)$ has index $\cov(q(H))=\cov(H)$ in the group $G/N$. By Proposition~\ref{p10.3}, $\is_{12}(q(H))=\sis_{123}(q(H))=d^*(H)=1/\cov(q(H))=1/\cov(H)$.
It can be shown that for each subset $A\subset G/N$ its preimage $q^{-1}(A)\subset G$ has densities
$\is_{12}(q^{-1}(A))=\is_{12}(A)$ and $d^*(q^{-1}(A))=d^*(A)$. In particular, for the subgroup $H=q^{-1}(q(H))$ we get
$\is_{12}(H)=\is_{12}(q(H))=1/\cov(H)$ and $d^*(H)=d^*(q(H))=1/\cov(H)$.
\end{proof}

Propositions~\ref{p10.1} and \ref{t10.4} imply:

\begin{proposition}\label{p10.6} For any group $G$ we get
$$[d^*=0]\subset [\sis_{123}{=}0]\subset[\wis_{12}{=}0]\subset [\is_{12}{=}0].$$
If the group $G$ is infinite, then the ideal $[d^*=0]$ contains all sets of cardinality $<|G|$ in $G$.
\end{proposition}

Next, we show that for any subset $A$ of a group $G$ with positive upper Banach density $d^*(A)$ there is an integer number $k$ dependent only on $d^*(A)$ such that the set $(A^{-1}A)^k$ is a subgroup of index $\le 1/d^*(G)$ in $G$.
Here for a subset $A\subset G$ its power $A^k\subset G$ is defined by induction: $A^1=A$ and $A^{k+1}=\{xy:x\in A^k,\;y\in A\}$ for $k\in\IN$.  We shall need the following fact proved in Lemma 12.3 of \cite{PB}.

\begin{proposition}\label{p11.1} If a symmetric subset $A=A^{-1}$ of a group $G$ has finite covering number $k=\cov(A)$, then the set $A^{4^{k-1}}$ is a subgroup of $G$.
\end{proposition}

Combining this proposition with Corollary~\ref{c8.2}, we get:

\begin{corollary}\label{c11.2} For any subset $A\subset G$ of positive upper Banach density $d^*(A)$ in a group $G$ and the number $k=\cov(AA^{-1}AA^{-1})\le1/d^*(A)$ the set $(AA^{-1})^{\frac124^{k}}$ is a subgroup of index $\le k$.
\end{corollary}

For partitions we can prove a bit more using Corollary~\ref{c8.3}.

\begin{corollary}\label{c11.3} For any partition $G=A_1\cup\dots\cup A_n$ of a group $G$ there is a cell $A_i$ of the partition such that the sets $(A_iA_i^{-1})^{4^{n-1}}$ is a subgroup of index $\le n$ in $G$.
\end{corollary}

\begin{proof} By Corollary~\ref{c8.3}, some cell $A_i$ of the partition has $\cov(A_iA_i^{-1})\le n=1/d^*(A_i)$ or $\cov\big((A_iA_i^{-1})^4\big)\le\cov\big((A_iA_i^{-1})^2\big)\le 1/d^*(A_i)<n$. In the first case $H=(A_iA_i^{-1})^{4^{n-1}}$ is a subgroup of $G$. In the second case $((A_iA_i^{-1})^4)^{4^{n-2}}=(A_iA_i^{-1})^{4^{n-1}}=H$ also is a subgroup of $G$. In both cases $H=(A_iA_i)^{4^{n-1}}$ is a subgroup of finite index $\cov(H)=\frac1{d^*(H)}\le \frac1{d^*(A_i)}\le n$.
\end{proof}

A subset $A$ of a group $G$ will be called a {\em shifted subgroup} if $A=xHy$ for some subgroup $H$ and some points $x,y\in G$. Observe that for a shifted subgroup $A$ the sets $AA^{-1}=xHx^{-1}$ and $A^{-1}A=y^{-1}Hy$ are subgroups conjugated to $H$ and $A=AA^{-1}xy=xyA^{-1}A$.

Corollary~\ref{c11.3} implies the following old result of Neumann \cite{Neumann}.

\begin{proposition}[Neumann]\label{c11.4} For any cover $G=A_1\cup\dots\cup A_n$ of a group $G$ by shifted subgroups some shifted subgroup $A_i$ has $\cov(A_i)\le n$.
\end{proposition}

\begin{proof} By Corollary~\ref{c11.3}, for some shifted subgroup $A_i$ the subgroup $A_i^{-1}A_i$ has index $\cov(A_i^{-1}A_i)\le n$. Since $A_i=xA_i^{-1}A_i$ for some $x\in G$, we conclude that $\cov(A_i)=\cov(xA_i^{-1}A_i)=\cov(A_i^{-1}A_i)\le n$.
\end{proof}

\section{Applications of the density $\is_{12}$ to $\IP^*$-sets}\label{s12n}

In this section we present an application of the density $\is_{12}$ to $\IP^*$ sets.
Following \cite[16.5]{HS98}, we call a subset $A$ of a group $G$ an {\em $\IP^*$-set} if for any sequence $(x_n)_{n\in\w}$ in $G$ there are indices $i_1<i_2<\dots <i_n$ such that $x_{i_1}x_{i_2}\dots x_{i_k}\in A$. By Theorem~5.12 of \cite{HS98}, any $\IP^*$-set $A\subset G$ belongs to every idempotent of the compact right-topological semigroup $\beta G$, and hence has a rich combinatorial structure, see \cite[\S14]{HS98}. The following theorem can be considered as a ``non-amenable'' generalization of Theorem~3.1 \cite{BHMc}.

\begin{proposition}\label{p12.1n} Let $G$ be a group endowed with a left-invariant Boolean ideal $\I\subset [\wis_{12}{=}0]$. If a set $A$ of a group $G$ has positive density $\is_{12}(A)>0$, then for every sequence $(x_i)_{i=1}^n$ of length $n> 1/\is_{12}(A)$ in $G$ there are two numbers $k<m\le n$ such that $x_{k+1}\dots x_m\in \Delta_\I(A)$. Consequently, $\Delta_\I(A)$ is an $\IP^*$-set.
\end{proposition}

\begin{proof} Consider the set $P=\{x_1\cdots x_k:1\le k\le n\}$. By Proposition~\ref{p5.2}, $\Ipack(A)\le 1/\is_{12}(A)<n$. Consequently there are two numbers $k<m\le n$ such that $x_1\dots x_kA\cap x_1\dots x_mA\notin \I$. The left invariance of the Boolean ideal $\I$ implies that $A\cap x_{k+1}\dots x_mA\notin \I$ and hence $x_{k+1}\dots x_m\in\Delta_\I(A)$.
\end{proof}

By analogy, we can use Proposition~\ref{p6.3} to prove:

\begin{proposition}\label{p12.2n} Let $G$ be a group endowed with a left-invariant Boolean ideal $\I\subset [\wus_{12}{=}0]$. If a set $A$ of a group $G$ has positive density $\us_{12}(A)>0$, then for every sequence $(x_i)_{i=1}^n$ of length $n> 1/\us_{12}(A)$ in $G$ there are two numbers $k<m\le n$ such that $x_{k+1}\dots x_m\in \Delta_\I(A)$. Consequently, $\Delta_\I(A)$ is an $\IP^*$-set.
\end{proposition}

Since any conjugacy-invariant set $A=A^{\conj[G]}=\bigcup_{x\in G}x^{-1}Ax$ has submeasure $\sis_{123}(A)=\is_{12}(A)$, Proposition~\ref{p12.1n} implies:

\begin{corollary}\label{c12.3n} Let $G$ be a group endowed with a left-invariant Boolean ideal $\I\subset [\wis_{12}{=}0]$. If a set $A$ of a group $G$ has positive density $\sis_{123}(A)>0$, then for every sequence $(x_i)_{i=1}^n$ of length $n> 1/\sis_{123}(A)$ in $G$ there are two numbers $k<m\le n$ such that $x_{k+1}\dots x_m\in \Delta_\I(A^{\conj[G]})$. Consequently, $\Delta_\I(A^{\conj[G]})$ is an $\IP^*$-set.
\end{corollary}

The subadditivity of the submeasure $\sis_{123}$ and Corollary~\ref{c12.3n} implies:

\begin{corollary} For any partition $G=A_1\cup\dots\cup A_n$ of a group $G$ endowed with a non-trivial left-invariant ideal $\I\subset[\wis_{12}{=}0]$, there is a cell $A_i$ of the partition such that $\Delta_\I(A^{\conj[G]})$ is an $\IP^*$-set.
\end{corollary}

\begin{remark} By Theorem 3.8 of \cite{BHMc}, the free group with two generators $F_2$ can be covered by two sets $A,B$ such that neither $AA^{-1}$ not $BB^{-1}$ is an $\IP^*$-set. This example shows that the free group $F_2$ admits no subadditive density $\mu:\PP(G)\to[0,1]$ such that $AA^{-1}$ is an $\IP^*$-set for any set $A\subset G$ of positive density $\mu(A)>0$.
\end{remark}

\section{Some Open Problems with Comments}\label{s12}

In this section we collect some problems related to Problems~\ref{prob1}--\ref{prob3}.

Motivated by Theorem~\ref{t2.1}, in \cite[Question F]{Erde}, J.~Erde asked whether, given a partition $\PP$ of an infinite group $G$ with $|\PP|<|G|$, there is $A\in\PP$ such that $\cov(AA^{-1})$ is finite.
The following extremely negatively answer to this question was obtained in \cite{b1}: Any infinite group $G$
admits a countable partition $G=\bigcup_{n\in\w}A_n$ such that $\cov(A_nA_n^{-1})\ge \mathrm{cf}(|G|)$ for each $n$.

\begin{problem}\label{p1n} Does each infinite group $G$ admit a countable partition $G=\bigcup_{n\in\w}A_n$ such that $\cov(A_nA_n^{-1})=|G|$ for all $n\in\w$?
\end{problem}

The answer to this problem is affirmative if the group $G$ is residually finite (in particular, Abelian or free), see \cite{b1}.
A stronger version of Problem~\ref{p1n} was considered in \cite{b2}.

\begin{question}
Does every infinite group $G$ admit a countable partition $G=\bigcup_{n<\w}A_n$ such that $\cov(A_n)=|G|$ for each $n\in\w$?
\end{question}

A subset $A\subset X$ of a $G$-space $X$ is called {\em $m$-thick} for a natural number $m$ if for each set $F\subset G$ of cardinality $|F|\le m$ there is a point $x\in X$ such that $Fx\subset A$.
A subset $A\subset G$ is {\em thick} if it is $m$-thick for every $m\in\IN$.
 Observe that a set $A\subset X$ is 2-thick if and only if $\Delta_\I(A)=G$ for the smallest ideal $\I=\{\emptyset\}$. The following proposition was proved in \cite[1.3]{BPS}.

\begin{proposition}\label{p13.1} For any partition $X=A_1\cup\dots\cup A_n$ of a $G$-space $X$ and any $m\in\IN$ there are a cell $A_i$ of the partition and a subset $F\subset G$ of cardinality $|F|\le m^{n-1}$ such that the set $FA_i$ is $m$-thick.
\end{proposition}

\begin{corollary} For any partition $X=A_1\cup\dots\cup A_n$ of a $G$-space $X$ there are a cell $A_i$ of the partition and a subset $F\subset G$ of cardinality $|F|\le 2^{n-1}$ such that $\Delta_\I(FA)=G$ for the smallest Boolean ideal $\I={\{\emptyset\}}$ on $X$.
\end{corollary}

\begin{corollary} For any partition $G=A_1\cup\dots\cup A_n$ of a group $G$ there is a cell $A_i$ of the partition such that $G=FA_iA_i^{-1}F^{-1}$ for some set $F\subset G$ of cardinality $|F|\le 2^{n-1}$.
\end{corollary}

\begin{problem}\label{prob12.4} Is it true that for each partition $G=A_1\cup\dots\cup A_n$ of a group $G$ there is a cell $A_i$ of the partition such that $G=FA_iA_i^{-1}F^{-1}$ for some set $F\subset G$ of cardinality $|F|\le n$.
\end{problem}

The following problem is stronger than Problem~\ref{prob12.4} but weaker than Problem~\ref{prob1}.

\begin{problem}\label{prob1a} Is it true that for any finite partition $G=A_1\cup\dots\cup A_n$ of a group $G$ there exist a cell $A_i$ of the partition and a subset $F\subset G\times G$ of cardinality $|F|\le n$ such that $G=\bigcup_{(x,y)\in F}xA_iA_i^{-1}y$?
\end{problem}

Another weaker version of Problem~\ref{prob1} also remains open:

\begin{problem} Let $G=A_1\cup\dots\cup A_n$ be a partition of a group $G$ such that $A_iA_j=A_jA_i$ for all indices $1\le i,j\le n$. Is there a cell $A_i$ of the partition with $\cov(A_iA_i^{-1})\le n$?
\end{problem}

Proposition~\ref{p13.1} contrasts with the following theorem proved in \cite{BPS}.

\begin{theorem}\label{t13.5} For every $k\in\IN$, any countable infinite group $G$ admits a partition $G=A\cup B$ such that for every $k$-element subset $K\subset G$ the sets $KA$ and $KB$ are not thick.
\end{theorem}

This theorem was proved with help of syndetic submeasures. A density $\mu:\PP(G)\to[0,1]$ on a group $G$ is called {\em syndetic} if for each subset $A\subset G$ with $\mu(A)<1$ and each $\e>\frac1{|G|}$ there is a subset $B\subset G\setminus A$ such that $\mu(B)<\e$ and $\cov(B)<\infty$. It can be shown that the density $\is_{12}$ is syndetic. According to Theorem 5.1 of \cite{BPS} (deduced from \cite{Weiss}), each countable group admits a left-invariant syndetic submeasure. This fact was crucial in the proof of Theorem~\ref{t13.5}.

\begin{problem} Does each group $G$ admit a left-invariant syndetic submeasure? Is the submeasure $\sis_{123}$ syndetic on each group $G$? Is the upper Banach density $d^*$ syndetic on each group $G$?
\end{problem}

Also we do not know if amenability of groups can be characterized via extremal densities or packing indices.

\begin{problem} Is a group $G$ amenable if for each partition $G=A_1\cup \dots\cup A_n$ there is a cell $A_i$ of the partition satisfying one of the conditions: \textup{(a)}~$\is_{12}(A_i)\ge\frac1n$, \textup{(b)}~$\pack(A_i)\le n$, \textup{(c)}~$\cov(A_iA_i^{-1})\le n$, \textup{(d)}~$\is_{12}(A_i)>0$, \textup{(e)}~$\pack(A_i)<\w$?
\end{problem}


\begin{thebibliography}{}

\bibitem{Ban} T.~Banakh, {\em Extremal densities and submeasures on groups}, preprint (http://arxiv.org/abs/1211.0717).

\bibitem{Ban2} T.~Banakh, {\em Extremal densities and measures on groups and $G$-spaces and their combinatorial applications}, Lecture Notes (http://arxiv.org/abs/1312.5078).

\bibitem{BF} T.~Banakh, M.~Fr\c aczyk, {\em On the structure of the difference sets in partitions of $G$-spaces and groups}, preprint.

\bibitem{BLR} T.~Banakh, N.~Lyaskovska, D.~Repovs, {\em Packing index of subsets in Polish groups}, Notre Dame J. Formal Logic., {\bf 50}:4 (2009) 453--468.

\bibitem{BL} T.~Banakh, N.~Lyaskovska, {\em Completeness of translation-invariant ideals in groups}, Ukr. Mat. Zh. {\bf 62}:8 (2010), 1022--1031.

\bibitem{BL11} T.~Banakh, N.~Lyaskovska, {\em Constructing universally small subsets of a given packing index in Polish groups}, Colloq. Math. {\bf 125} (2011) 213--220.


\bibitem{BPS} T.~Banakh, I.V.~Protasov, S.~Slobodianiuk, {\em Syndetic submeasures and partitions of G-spaces and groups}, IJAC, {\bf 23}:7 (2013), 1611--1623.

\bibitem{BRS}  T.~Banakh, O.~Ravsky, S.~Slobodianiuk, {\em On partitions of $G$-spaces and $G$-lattices}, preprint (http://arxiv.org/abs/1303.1427).

\bibitem{BHMc} V.~Bergelson, N.~Hindman, R.~McCutcheon, {\em Notions of size and combinatorial properties of quotient sets in semigroups}, Topology Proc. {\bf 23} (1998), 23--60.

\bibitem{Erde} J.~Erde, {\em A Note on Combinatorial Derivation}, preprint (http://arxiv.org/abs/1210.7622).

\bibitem{HS98} N.~Hindman, D.~Strauss, {\em Algebra in the Stone-\v Cech
 compactification}, de Gruyter, Berlin, New York, 1998.

\bibitem{HS} N.~Hindman, D.~Strauss, {\em Density in arbitrary semigroups}, Semigroup Forum {\bf 73}:2  (2006), 273--300.



\bibitem{Lambert} R.M.~Corless, G.H.~Gonnet, D.E.G.~Hare, D.J.~Jeffrey, D.E.~Knuth, {\em On the Lambert $W$ function}, Adv. Comput. Math. {\bf 5}:4 (1996), 329--359.

\bibitem{Kourovka} V.D.~Mazurov, E.I.~Khukhro, (eds.) {\em Unsolved problems in group theory: the Kourovka notebook}, Thirteenth augmented edition. Russian Academy of Sciences Siberian Division, Institute of Mathematics, Novosibirsk, 1995. 120 pp.

\bibitem{Neumann} B.H.~Neumann, {\em Groups covered by permutable subsets}, J. London Math. Soc. {\bf 29} (1954), 236--248.

\bibitem{Pat} A.~Paterson, {\em Amenability}, Mathematical Surveys and Monographs, 29. American Mathematical Society, Providence, RI, 1988.

\bibitem{Prot} I.~Protasov, {\em Combinatorial Derivation}, preprint (http://arxiv.org/abs/1210.0696).

\bibitem{b3} I.~Protasov, {\it Partitions of groups onto large subsets}, Math. Notes {\bf 73} (2003), 271--281.

\bibitem{PB} I.~Protasov, T.~Banakh, {\em Ball stuctures and colorings of graphs and groups}, VNTL Publ. 2003, 148p.

\bibitem{b1} I.~Protasov, S.~Slobodianiuk, {\it A conjecture on partition of groups}, preprint.

\bibitem{b2} I.~Protasov, S.~Slobodianiuk, {\it A note on partition of groups}, preprint (to appear in Arxiv).

\bibitem{factorial} H.~Robbins, {\em A remark on Stirling's formula}, Amer. Math. Monthly {\bf 62} (1955), 26--29.

\bibitem{Slob} S.~Slobodianiuk, Unpublished note, (2012).

\bibitem{Sol}  S.~Solecki, {\em Size of subsets of groups and Haar null sets}, Geom. Funct. Anal. {\bf 15}:1 (2005), 246--273.

\bibitem{Weiss} B.~Weiss, {\em Minimal models for free actions}, in: Dynamical Systems and Group Actions (L.~Bowen, R.~Grigorchuk, Ya.~Vorobets eds.), Contemp. Math. {\bf 567}, Amer. Math. Soc. Providence, RI, (2012),  249--264.

\end{thebibliography}
\end{document}